\documentclass[a4paper]{amsart}

\usepackage{amsthm, amsfonts, amssymb, amsmath, latexsym, enumitem,array}
\usepackage[all]{xy}
\usepackage[utf8]{inputenc}
\usepackage{hyperref,cite,mdwlist}
\usepackage{mathrsfs}
\usepackage{tikz}
\usepackage{libertine}
\usepackage[libertine,scaled=0.9]{newtxmath}
\usepackage{dsfont}
\usepackage{soul}

\usepackage {color}

\newtheorem{thm}{Theorem}
\newtheorem{lem}{Lemma}
\newtheorem{corol}{Corollary}
\newtheorem{claim}{Claim}

\theoremstyle{definition}
\newtheorem{rmk}{Remark}

\newcommand{\cE}{\mathcal{E}}
\newcommand{\cK}{\mathcal{K}}
\newcommand{\cL}{\mathcal{L}}

\newcommand{\cF}{\mathcal{F}}

\newcommand{\cV}{\mathcal{V}}

\newcommand{\cI}{\mathcal{I}}

\newcommand{\cU}{\mathcal{U}}
\newcommand{\cO}{\mathcal{O}}
\newcommand{\cHom}{\mathcal{H}om}
\newcommand{\cExt}{\mathcal{E}xt}

\DeclareMathOperator{\Spl}{\mathrm{Spl}}

\DeclareMathOperator{\cok}{coker}

\DeclareMathOperator{\rk}{rk}
\DeclareMathOperator{\Ext}{Ext}

\DeclareMathOperator{\Hom}{Hom}
\DeclareMathOperator{\End}{End}

\DeclareMathOperator{\Hilb}{Hilb}

\DeclareMathOperator{\HH}{H}
\DeclareMathOperator{\hh}{h}
\DeclareMathOperator{\rM}{M}
\newcommand{\rR}{\boldsymbol{R}}

\DeclareMathOperator{\ext}{ext}
\def\p{p}

\newcommand{\ZZ}{\mathbb Z}

\newcommand{\QQ}{\mathbb Q}

\newcommand{\PP}{\mathbb P}

\newcommand{\bk}{{\boldsymbol k}}

\newcommand{\Pic}{\mathrm{Pic}}

\newcommand{\mono}{\hookrightarrow}
\newcommand{\epi}{\twoheadrightarrow}

\newcommand{\sm}{\mathrm{sm}}
\newcommand{\rp}{\mathrm{p}}
\newcommand{\rP}{\mathrm{P}}
\newcommand{\rN}{\mathrm{N}}
\newcommand{\rsta}{\mathrm{N}}

\allowdisplaybreaks[1]
\begin{document}

\begin{abstract}
We show that any polarized K3 surface supports special Ulrich bundles of rank $2$.
\end{abstract}

\setul{1pt}{1pt}

%\linenumbers
%\linenumberdisplaymath

%%%%%%%%%%%%%%%%%%%%%%%%%%%%%%%%%%%%%%%%%%%%%%%%%%%%%%%%%%%%%%%%%%%%%%%%%%%%%%
%%%%%%%%%%%%%%%%%%%% Author(s) and Address %%%%%%%%%%%%%%%%%%%%%%%%%%%%%%%%%%%
%%%%%%%%%%%%%%%%%%%%%%%%%%%%%%%%%%%%%%%%%%%%%%%%%%%%%%%%%%%%%%%%%%%%%%%%%%%%%%

\title{Ulrich bundles on K3 surfaces}

\author{Daniele Faenzi}
\email{daniele.faenzi@u-bourgogne.fr}
\address{Institut de Mathématiques de Bourgogne,
UMR CNRS 5584,
Université de Bourgogne et Franche-Comté,
9 Avenue Alain Savary,
BP 47870,
21078 Dijon Cedex,
France}
\thanks{Author 
  partially supported by ISITE-BFC project (contract ANR-lS-IDEX-OOOB)}
\keywords{ACM vector sheaves and bundles, Ulrich sheaves, K3 surfaces}
\subjclass[2010]{14F05; 13C14; 14J60}

\maketitle

%%% \begin{center}
%%%  {\textit{Work in progress, last edited \today.}}  
%%% \end{center}
%%% 
%%% 
%%% \medskip

Given an $n$-dimensional closed subvariety $X \subset
\PP^N$, a coherent sheaf $\cF$ on $X$ is Ulrich if
$\HH^*(\cF(-t))=0$ for $1 \le t \le n$.
We refer to \cite{coskun:ulrich-survey,beauvill:ulrich-intro} for
an introduction. We mention that Ulrich sheaves are related to
Chow forms (this was the main motivation for they study in
\cite{eisenbud-schreyer-weyman}), to determinantal representations and
generalized Clifford algebras, to Boij-Söderberg theory (cf.
\cite{eisenbud-schreyer:betti-numbers}) to the minimal resolution
conjecture, and to the representation
type of varieties (cf. \cite{daniele-joan:wild}).

Conjecturally, Ulrich sheaves exist for any $X$, see
\cite{eisenbud-schreyer-weyman}.  They are known to exist for
several classes of varieties e.g. complete intersections, curves, Veronese,
Segre, Grassmann varieties. 
Low-rank Ulrich bundles on surfaces have been studied
intensively, and Ulrich bundles of rank $2$ (or sometimes $1$) are
known in many cases. We refer to
\cite{casnati:special-non-special,beauvill:ulrich-intro} for a survey
and further references. Let us only review some of the cases that are most relevant for us, namely among
surfaces with  
trivial canonical bundle.

In \cite{beauville:abelian-ulrich}, Ulrich bundles of rank $2$ are
proved to exist on abelian surfaces.
In \cite{aprodu-ortega-farkas:crelle}, it is proved that K3 surfaces
support Ulrich bundles 
of rank $2$, provided that some Noether-Lefschetz open condition is
satisfied. The case of quartic surfaces was previously analyzed in
detail in \cite{coskun-kulkarni-mustopa:quartic}.
The main techniques used so far are the Serre construction starting
from points on $X$ and Lazarsfeld-Mukai bundles.

\medskip

In this note, we show that any K3 surface supports an Ulrich bundle $\cE$ of
rank $2$ with $c_1(\cE)=3H$, for any polarization $H$.
So these bundles are \textit{special}, cf.
\cite{eisenbud-schreyer-weyman}.
 We allow singular surfaces with trivial canonical bundle.
The main tool is an enhancement of Serre's construction based
on unobstructedness of simple sheaves on a $K3$ surface.

\medskip
Let us state the result more precisely.
We work over an algebraically closed field $\bk$.
Let $X$ be an integral (i.e. reduced and irreducible) projective surface with $\omega_X \simeq
\cO_X$ and $\HH^1(\cO_X)=0$. We denote by $X_{\sm}$ the smooth locus
of $X$.

Fix a very ample divisor $H$ on $X$. Under the closed embedding given by the complete linear
series $|\cO_X(H)|$ we may view $X$ as a subvariety of some projective
space $\PP^g$. A  hyperplane section $C$ of $X$ is a
projective Gorenstein curve of arithmetic genus $g$ with $\omega_C \simeq
\cO_C(H)$, where $H$ also denotes the restriction of $H$ to $C$. We
may choose $C$ to be integral too.

\medskip

A locally Cohen-Macaulay sheaf $\cE$ on $X$ is arithmetically
Cohen-Macaulay (ACM) 
if $\HH^1(\cE(tH))=0$  for all $t \in \ZZ$. A special class of ACM
sheaves are Ulrich sheaves, which are characterized by the property
$\HH^*(\cE(-tH))=0$ for $t=1,2$.
%Sometimes these integers are required to be $1$ and $2$ in which
%case we speak of an initialized Ulrich sheaf.
Of course all these notions depend on the polarization $H$.
We call simple a sheaf whose only
endomorphisms are homotheties.

\medskip

\begin{thm} \label{main}
  Let $X$ and $H$ be as above. Then there exists a simple Ulrich vector
  bundle of rank $2$ on $X$ whose determinant is $\cO_X(3H)$.
\end{thm}

The strategy to prove the theorem is the following. First we build an ACM
vector bundle $\cE$ of rank $2$ by Serre's construction applied to a
projective coordinate system in $X$. Then we perform an elementary
modification of $\cE$ along a single generic point $\p \in X$, producing a
simple non-reflexive sheaf having the Chern character of an Ulrich
bundle. Finally we flatly deform such sheaf and check
that generically this yields the desired Ulrich bundle.

\medskip

Prior to all this, we start by observing that the trivial
bundle is a (trivial) example of ACM line bundle. Indeed,
using that $\HH^1(\cO_X)=0$ and that $C$ is connected, one checks that $\HH^1(\cO_X(-H))=0$. In turn,
this easily implies $\HH^1(\cO_X(-tH))=0$ for all $t \ge 2$. Also,
Serre duality and triviality of $\omega_X$ give
$\HH^1(\cO_X(tH))=0$ for all $t \ge 0$. This way, we see that $\cO_X$ is an ACM
line bundle on $X$. Combining this with Max Noether's theorem on the
generation of the canonical ring of curves 
(cf. \cite{rosenlicht:equivalence} for a version for Gorenstein curves)
one obtains, working as in \cite[Theorem 6.1]{saint-donat:K3}, that $X\subset \PP^g$ is an ACM surface of degree $2g-2$.

However this line bundle is never Ulrich, nor is any line bundle of
the form $\cO_X(dH)$. So
generically (for instance when $X$ has Picard number $1$) the surface $X$ will not support Ulrich line bundles.
We thus move to rank two and start by constructing a simple ACM bundle.

\begin{lem} \label{ACM1}
  Let $Z \subset X_{\sm}$ be a set of $g+2$ points in general linear
  position. Then there is a unique coherent sheaf $\cE$ of rank
  $2$ fitting into a non-splitting exact sequence:
  \begin{equation} \label{IZ}
  0 \to \cO_X \to \cE \to \cI_Z(H) \to 0.    
  \end{equation}
  The sheaf $\cE$ is locally free, simple and ACM. It satisfies:
  \[\cE\simeq
  \cE^*(H), \qquad \hh^0(\cE)=1, \qquad \hh^1(\cE)=\hh^2(\cE)=0\qquad \ext_X^1(\cE,\cE)=2g+4.
  \]
\end{lem}

\begin{proof}
  Taking cohomology of the exact sequence
  \begin{equation}
    \label{IO}
  0 \to \cI_Z(H) \to \cO_X(H) \to \cO_Z \to 0,    
  \end{equation}
  and using the fact that $Z$ is in general linear position and hence
  contained in no hyperplane, we get $\HH^0(\cI_Z(H))=0$ and
  $\hh^1(\cI_Z(H))=1$.

  By Serre duality we get $\ext^1_X(\cI_Z(H),\cO_X)=\hh^1(\cI_Z(H))=1$
  so, up to proportionality, there is a unique non-splitting extension
  of the desired form. Correspondingly, there exists a unique coherent sheaf $\cE$ of rank two
  fitting into a non-splitting exact sequence of the form \eqref{IZ}.
  The sheaf $\cE$ we obtain this way satisfies $\hh^0(\cE)=1$ and
  $\HH^1(\cE) \simeq \Ext^1_X(\cE,\cO_X)^*=0$ because applying
  $\Hom_X(-,\cO_X)$ to \eqref{IZ} we obtain a non-zero map (and thus
  an isomorphism) 
  $\HH^0(\cO_X) \to \Ext^1_X(\cI_Z(H),\cO_X)$.

  This map is the dual of the homomorphism
  $\HH^1(\cI_Z(H)) \to \HH^2(\cO_X)$ obtained by taking global sections
  in \eqref{IZ}. So $\HH^1(\cE)=\HH^2(\cE)=0$.
  
  If $X$ is smooth we deduce that $\cE$ is locally free from the
  Cayley-Bacharach property, cf. for instance
  \cite[Theorem 5.1.1]{huybrechts-lehn:moduli}. Indeed, since $Z$ is
  in general linear position (i.e. $Z$ is a projective frame in
  $\PP^g$), no hyperplane passes through any subset of $g+1$ points of
  $Z$. Anyway the statement follows in general by a minor modification of the
  argument appearing in
  \cite[Lemma 7.2]{daniele-joan:wild}. Indeed by the local-to-global
  spectral sequence, using $\HH^1(\cO_X(-H))=0$ and
  $\cHom_X(\cI_Z(H),\cO_X) \simeq \cO_X(-H)$ we get the following exact sequence:
  \[
  0 \to \Ext_X^1(\cI_Z(H),\cO_X) \to \HH^0(\cExt_X^1(\cI_Z(H),\cO_X))
  \to \HH^2(X,\cO_X(-H)) \to 0.
  \]
  In turn, using $\cExt_X^1(\cI_Z(H),\cO_X) \simeq \omega_Z\simeq
  \cO_Z$ and $\HH^2(X,\cO_X(-H)) \simeq \HH^0(X,\cO_X(H))^*$, 
  if we choose $Z$ to be a projective coordinate
  system of $\PP^g$, we rewrite this exact sequence as:
  \[
  0 \to \Ext_X^1(\cI_Z(H),\cO_X) \to \HH^0(\cO_Z) \xrightarrow{
  \begin{pmatrix}
    1 & \cdots & 0 & 1 \\
    \vdots & \ddots & \vdots & \vdots   \\
    0 & \cdots & 1 & 1
  \end{pmatrix}} \HH^0(X,\cO_X(H))^* \to 0.
  \]
  So $\Ext_X^1(\cI_Z(H),\cO_X)$ is generated by the vector $(1,\ldots,1,-1)^t$ and
  since this vector corresponds to an extension in
  $\cExt_X^1(\cI_Z(H),\cO_X)$ which is non-zero at any point of $Z$ we have that
  the sequence defining $\cE$ is locally non-split around each point of
  $Z$, which in turn implies that $\cE$ is locally free at each such
  point (and hence everywhere). From
  $c_1(\cE)=H$, since $\cE$ is locally free of rank $2$, we get a canonical isomorphism $\cE \simeq \cE^*(H)$.

  Let us prove that $\cE$ is ACM. We already have $\hh^1(\cE)=0$ and
  thus by Serre duality
  $\hh^1(\cE(-H))=\hh^1(\cE^*(H))=\hh^1(\cE)=0$. Also
  $\hh^0(\cE(-H))=0$ and $\hh^2(\cE(-H))=1$. Note that, choosing an integral
  hyperplane section curve $C$ that avoids
  $Z$, \eqref{IZ} becomes:
  \[
  0 \to \cO_C \to \cE|_C \to \cO_C(H) \to 0.
  \]
  From $\HH^k(\cE(-H))=0$ for $k=0,1$ we deduce $\hh^0(\cE|_C)=1$ so
  the previous exact sequence does not split. Then
  $\hh^0(\cE|_C(-H))=0$. This easily implies $\HH^1(\cE(-2H))=0$ and
  actually $\HH^1(\cE(-tH))=0$ for all $t \ge 2$. Serre duality now
  gives  $\HH^1(\cE(tH))=0$ for all $t \ge 1$. In other words $\cE$ is ACM.
  
  It remains to check that $\cE$ is simple. Applying $\Hom_X(\cE,-)$ to
  the exact sequence \eqref{IO} we get that the non-zero space 
  $\Hom_X(\cE,\cI_Z(H))$ is
  contained in $\Hom_X(\cE,\cO_X(H)) \simeq \HH^0(\cE) \simeq \bk$, so 
  $\hom_X(\cE,\cI_Z(H))=1$.
  As $\cHom_X(\cE,\cO_Z)$ is a skyscraper sheaf of rank $2$ at $Z$ we 
  have $\ext^k_X(\cE,\cO_Z)=(2g+4)\delta_{0,k}$. We deduce 
  $\ext^1_X(\cE,\cI_Z(H))=2g+4$ and   $\ext^0_X(\cE,\cI_Z(H))=0$.
  
  Therefore, applying $\Hom_X(\cE,-)$ to
  the \eqref{IZ}, since $\Hom_X(\cE,\cO_X) \simeq \hh^2(\cE)=0$ we get
  that $\End_X(\cE)$ is contained in $\Hom_X(\cE,\cI_Z(H))$ and is
  therefore $1$-dimensional. This says that $\cE$ is simple. By Serre
  duality $\ext^2_X(\cE,\cE)=1$.
  We deduce $\ext^1_X(\cE,\cE)=\ext^1_X(\cE,\cI_Z(H))=2g+4$.
\end{proof}

Given a reduced subscheme $Z \in
\Hilb_{g+2}(X_\sm)$ consisting of points in general linear position, there
is a unique rank-$2$ bundle associated with $Z$ according to the
previous lemma. We denote it by $\cE_Z$. We write 
$\cO_\p$ for the skyscraper sheaf of a point $\p \in X$.

\begin{lem} \label{issimple}
  Assume $\eta : \cE_Z \to \cO_\p$ is surjective. Then $\cE^\eta =
  \ker(\eta)$ is a simple sheaf with:
  \[c_1(\cE^\eta)=H, \qquad c_2(\cE^\eta)=g+3, \qquad \ext^1_X(\cE^\eta,\cE^\eta)=2g+8.
  \]
\end{lem}

\begin{proof}
  Recall that $\cE=\cE_Z$ is simple and observe that this implies
  $\Hom_X(\cE,\cE^\eta)=0$, as the composition of any non-zero map $\cE
  \to \cE^\eta$ with $\cE^\eta \mono \cE$ would provide a self-map of
  $\cE$ which is not a multiple of the identity. Also, since $\cE$ is
  locally free we have $\hom_X(\cE,\cO_\p) = 2$ and
  $\Ext^k_X(\cE,\cO_\p)=0$ for $k > 0$. Therefore, using Lemma \ref{ACM1} and
  applying $\Hom_X(\cE,-)$ to the exact sequence:
\begin{equation}
  \label{Eeta}
  0 \to \cE^\eta \to \cE \to \cO_\p \to 0.  
\end{equation}
  we obtain $\ext^1_X(\cE,\cE^\eta)=2g+5$ and $\ext^2_X(\cE,\cE^\eta)=1$. 

  Next,
  Serre duality gives $\ext^k_X(\cO_\p,\cE)=2\delta_{2,k}$, while
  $\ext^k_X(\cO_\p,\cO_\p)$ is the dimension of the $k$-th exterior
  power of the normal bundle of $\p$ in $X$ and thus takes value
  $2\choose k$. Therefore, applying $\Hom_X(\cO_\p,-)$ to \eqref{Eeta}
  we find
  $\ext^1_X(\cO_\p,\cE^\eta)=1$ and $\ext^2_X(\cO_\p,\cE^\eta)=3$.
  Putting these computations together and applying
  $\Hom_X(-,\cE^\eta)$ again to \eqref{Eeta} we get:
  \[
  \hom_X(\cE^\eta,\cE^\eta) = \ext^2_X(\cE^\eta,\cE^\eta)=1, \qquad \ext^1_X(\cE^\eta,\cE^\eta)=2g+8.
  \]
  The computation of Chern classes is straightforward.
\end{proof}

\begin{lem} \label{surietta}
  Let $\p \in X_\sm \setminus Z$. Then, for a generic map $\eta : \cE_Z
  \to \cO_\p$, the induced map on global sections $\HH^0(\eta) : \HH^0(\cE_Z) \to \HH^0(\cO_\p)$ is an isomorphism.
\end{lem}

\begin{proof} Put $\cE=\cE_Z$.
  It suffices to check
  that there exists $\eta$ such that the induced map $\HH^0(\eta) : \bk \simeq \HH^0(\cE)
  \to \HH^0(\cO_\p) \simeq \bk$ is an isomorphism, for this is an open
  condition. To do it, we apply $\Hom_X(\cI_Z(H),-)$ to the
  exact sequence:
  \[
  0 \to \cI_\p \to \cO_X \to \cO_\p \to 0.
  \]
  This gives an exact sequence:
  \[
    \Ext^1_X(\cI_Z(H),\cI_\p) \to \Ext^1_X(\cI_Z(H),\cO_X) \to
    \Ext^1_X(\cI_Z(H),\cO_\p).
  \]
  
  Observe that $\cHom_X(\cI_Z(H),\cO_\p)\simeq \cO_\p$ and
  $\cExt^1_X(\cI_Z(H),\cO_\p) = 0$ as these sheaves are
  computed locally on $X$ and, since $\p \cap
  Z = \emptyset$, we may choose an open cover of $X$ consisting of
  subsets where $\cI_Z$ is
  trivial or $\cO_\p$ vanishes. Then the local-to-global spectral sequence
  gives $\Ext^1_X(\cI_Z(H),\cO_\p)=0$ so the extension corresponding
  to \eqref{IZ} admits a lifting to $\cI_\p$. In other words, we get the commutative exact
  diagram:
  \[
  \xymatrix@-2ex{
    & 0 \ar[d] & 0 \ar[d] \\
    0 \ar[r] & \cI_\p \ar[d] \ar[r] & \cE^\eta \ar[r] \ar[d] & \cI_Z(H) \ar[r] \ar@{=}[d] & 0\\
    0 \ar[r] & \cO_X \ar[r] \ar[d] & \cE \ar[r] \ar^-\eta[d] & \cI_Z(H) \ar[r] & 0 \\
    & \cO_\p \ar@{=}[r] \ar[d] & \cO_\p \ar[d]\\ 
    & 0 & 0 \\ 
    }
  \]
  where  $\eta$ and $\cE^\eta$ are defined by the diagram.
  For this choice of $\eta$ we get, by the top row of the diagram,
  $\HH^0(\cE^\eta)=0$, which implies that $\HH^0(\eta)$ is an isomorphism.
\end{proof}

By the previous lemma, we may  choose $\cE_Z$ as in Lemma
\ref{ACM1}, a point $\p \in X_\sm \setminus Z$, some $\eta
: \cE_Z \epi \cO_\p$ and consider the sheaf $\cE^\eta$. The goal is
to deform $\cE^\eta(H)$ to an Ulrich bundle.
We use the notation  $\cF_s^*$ for $(\cF_s)^*$ (which is a
priori not the same as $(\cF^*)_s$). 

\begin{lem} \label{deform}
  There exist a smooth connected variety $S_0$ of dimension $2g+8$ and a flat family of simple
  sheaves $\cF$ on $X \times S_0$
  such that
   $\cF_s(H)$ is an Ulrich bundle for $s$ generic in $S_0$ and $\cF_{s_0}
   \simeq \cE^\eta$ for some distinguished point $s_0$ of $S_0$.
\end{lem}

\begin{proof}
We proved in Lemma \ref{issimple} that $\cE^\eta$ is simple.  
Since the non-locally free locus of $\cE^\eta$ is disjoint from the
singular locus of $X$, we may apply the arguments of \cite[Theorem
0.1]{mukai:symplectic}.
In particular (cf. \cite{altman-kleiman:compactifying}) the moduli functor of simple
sheaves on $X$ is pro-represented by a moduli space $\Spl_X$ which can be
constructed in the étale topology and which is smooth of
dimension $2g+8$ at $\cE^\eta$ (this is essentially \cite[Theorem
0.3]{mukai:symplectic}).
Therefore there exists an open piece of $\Spl_X$ which is a quasi-projective variety $S$ equipped with a flat family $\cF$ of simple
sheaves on $X$, such that the 
induced map $S \to \Spl_X$ is a local isomorphism 
around the point corresponding to $\cE^\eta$. We denote by $s_0$ this
point, so that  $\cF_{s_0} \simeq
\cE^\eta$.

We
may assume that $S$ is smooth and connected of dimension $2g+8$. 
Since the reflexive hull $\cE$ of $\cE^\eta$ is locally free and satisfies
the assumption of \cite[Corollary 1.5]{artamkin:torsion-free}, we get
that $\cF_s$ is locally free for all $s$ in an open dense subset $S_1$
of $S$. 

Now observe that $\HH^*(\cF_{s_0})=0$ by Lemmas \ref{ACM1} and
\ref{surietta}. Then, semicontinuity ensures that $\HH^*(\cF_s)=0$ for
all $s$ in an open dense
subset $S_0$ of $S_1$. Therefore, the isomorphism $\cF_s^* \simeq
\cF_s(-H)$ and Serre duality give $\HH^i(\cF_s(-H)) \simeq
\HH^{2-i}(\cF_s^*(H))^* \simeq \HH^{2-i}(\cF_s)^*=0$.
This says that $\cF_s(H)$ is a special Ulrich bundle, for all $s \in S_0$.
\end{proof}

For the reader's benefit we also provide a 
proof of Lemma \ref{deform} independent of
\cite{artamkin:torsion-free}. 
The point is to check 
that $\cF_s$ is locally free for all $s$ in an open dense subset of $S$. 
To do this, first recall again that the non-locally free locus of $\cE^\eta$
is disjoint from the singular locus of $X$, so up to shrinking $S$ we may assume that this happens for
$\cF_s$ for all $s \in S$. Then $\cF_s^{**}$ is locally free for $s \in S$. 

Next, we may find an integer
$t_0 \le -1$ such that $\HH^0(\cF_s^{**}(t_0H))=\HH^1(\cF_s^{**}(t_0H))=0$ for
all $s \in S$. This can be done for instance using Kollar's theory
of husks (cf. \cite{kollar:hulls-husks}), which gives a stratification
$(S_i)_{i=1,\ldots,r}$ of $S$ such
that $\cF_s^{**}$ 
defines a flat family of sheaves on $X$ parametrized by $S_i$. Using
base change over each each $S_i$ one finds $t_i$ satisfying the required
vanishing together with $\HH^0(\cF_s^{**}(t_iH)|_C)=0$, for a fixed
curve $C \in |\cO_X(H)|$. Then $t_0$ can be taken to be the minimum
among $t_1,\ldots,t_r$.

Recall that $\HH^*(\cF_{s_0})=0$ and observe that \eqref{Eeta} gives:
\[
\hh^1(\cF_{s_0}(tH))=\left\{
  \begin{array}[h]{ll}
    1 & \mbox{if $t \le -1$}, \\
    0 & \mbox{if $t \ge 0$}.
  \end{array}
\right.
\]
By semi-continuity, we have thus $\HH^*(\cF_s)=0$, $\hh^1(\cF_s(tH))=0$ for all $t \ge 0$
and $\hh^1(\cF_s(tH))\le 1$ for $t \le -1$ for all $s$ in an open
dense subset of $S$. We still call $S$ this subset. 

Next, for all $s \in S$ we consider the double dual sequence:
\begin{equation}
  \label{tau}
0 \to \cF_s \to \cF_s^{**} \to \tau(\cF_s) \to 0.
\end{equation}
where the torsion sheaf $\tau(\cF_P)$ is defined by the sequence. Put
$\ell_s$ for the length of $\tau(\cF_s)$. 

Since $\HH^0(\cF_s^{**}(t_0H))=\HH^1(\cF_s^{**}(t_0H))=0$, from the
previous exact sequence we get 
$\ell_s=\hh^0(\tau(\cF_s))=\hh^1(\cF_s(t_0H)) \le 1$ (we neglect to
indicate the
twist on zero-dimensional sheaves).

Now we have two alternatives. Namely, either for $s$ general enough in $S$
one has $\ell_s =
0$, i.e. $\tau(\cF_s)=0$; or otherwise for all $s \in S$ we get $\ell_s=1$, i.e. $\tau(\cF_s) \simeq
\cO_{\p_s}$, for some
point $\p_s \in X$ with $\p_{s_0}=\p$. 

In the first case, we have
$\cF_s \simeq \cF_s^{**}$ and $\cF_s$ is locally free. So we would like to rule out the second alternative. By contradiction we assume that, for
all $s \in S$, we have $\tau(\cF_s) \simeq \cO_{\p_s}$. This gives a map $\gamma : S \to X$ associating
$\p_s$ to $s$. This time $\cF^{**}$
is flat over $S$ and  \eqref{tau} is the
restriction to $X \times \{s\}$ of a sequence on $X \times S$:
\[
0 \to \cF \to \cF^{**} \to \tau(\cF) \to 0,
\]
with $(\cF_s)^{**} \simeq (\cF^{**})_s$ and where $\tau(\cF)$ is 
a line bundle supported on the graph of $\gamma$.

Also, again the previous exact sequence together with $\HH^*(\cF_s)=0$
gives $\hh^0(\cF^{**}_s)=1$ so that
$\cF_s^{**}$ has a unique non-zero global section up to a scalar. This section
vanishes along a subscheme $Z_s \subset X$ and, up to shrinking again
$S$ we may 
assume that $Z_s$ is zero-dimensional reduced and in general linear
position, because 
these are open conditions, so that
$\cF^{**}_s \simeq \cE_{Z_s}$.

For each sheaf $\cF^{**}_s$ of
this family, we denote by $\eta_s : \cF^{**}_s \epi \cO_{\p_s}$ the induced
surjection of $\cF_s^{**}$ onto $\tau(\cF_s)$. We think of $\eta_s$ as an element of
$\PP(\HH^0(\cF^{**}_s|_{\p_s})) \simeq \PP^1$ (we adopt the convention
of writing $\PP(V)$ for
the projective space of hyperplanes of a vector space $V$). Plainly, we have $\cF^{**}_{s_0}
\simeq \cE^\eta$, $\tau(\cF_{s_0}) \simeq \cO_{\p}$ and $\eta_{s_0}$ is
identified with $\eta$. Note that $\cF_s = \ker(\eta_s)$.

We assert that the family $\cF$
is parametrized by an open subset $T$ of the set of triples:
\[
\{
(W,q,\xi) \mid W \in \Hilb_{g+2}(X), \, q \in X, \, \xi \in \PP(\HH^0(\cE_W|_{q}))
\}.
\]

The subset $T$ consists of $(W,q,\xi)$ with $W
\subset X_\sm$ is reduced and in 
general linear position in $X$, $q \in X_\sm \setminus W$ and $\xi$
is surjective. Given such a triple, we get the sheaf $\ker(\xi)$ which is
simple by Lemma \ref{issimple}.
Clearly this gives a flat deformation of $\cE^\eta$ so, because $S
\to \Spl_X$ is a local
isomorphism at $\cE^\eta$, there is a
possibly smaller open subset $T_0$ such that all the resulting sheaves
$\ker(\xi)$ 
are of the form $\cF_s$, for some $s \in S$. By construction any
sheaf $\cF_s$ should be of this form by taking $q=\p_s$, $W=Z_s$ and $\xi=\eta_s$.

But $T_0$ is an open dense subset of a $\PP^1$-bundle over an open
subset of $\Hilb_{g+2}(X) \times X$ and thus has dimension
$1+2(g+2)+2=2g+7$. Therefore $T_0$ cannot dominate $S$, as $\dim(S)=2g+8$.
This says that the second alternative does not take place, so we have
proved that $\cF_s(H)$ is an Ulrich bundle for general $s$.

\bigskip

Recall the notation $\rM_X(v)$ for the moduli space of $H$-semistable
sheaves $\cF$ on $X$ whose Mukai vector $v=(v_0,v_1,v_2)$ satisfies
$v_0=\rk(\cF)$,  $v_1=c_1(\cF)$ and $v_2=\chi(\cF)-\rk(\cF)$.
From \cite[Lemma 2.1]{qin:simple} we obtain the following
stronger version of Theorem \ref{main}.

\begin{corol} \label{stable-rank-2}
  If $X$ is smooth, $\rM_X(2,H,-2)$ is of dimension $2g+8$ and a general point
  of it corresponds to a sheaf $\cE$ which is stable (with respect to
  all polarizations) and such that $\cE(H)$ is a
  special Ulrich bundle.
\end{corol}

Again, we also offer a proof independent of \cite{qin:simple,artamkin:torsion-free}.
Consider the family of Ulrich sheaves $\cF(H)$ with parameter space $S_0$
constructed in the previous lemma. 
Recall that, for generic $s \in S_0$, the sheaf $\cF_s(H)$ is Ulrich,
hence semistable with Ulrich
sheaves as Jordan-Hölder factors (cf. \cite[Lemma 7.1]{daniele-joan:wild}). 
So we have to check that $\cF_s$ is not strictly
semistable. If it was, we would have an exact sequence:
\begin{equation}
  \label{extension}
0 \to \cL \to \cF_s \to \cL^*(H) \to 0,  
\end{equation}
where $\cL(H)$ is an Ulrich sheaf or rank $1$ on $X$. Actually $\cL(H)$ is
an Ulrich line bundle since $X$ is smooth. Since $\cL$ and $\cL^*(H)$
are rigid in view of $\HH^1(\cO_X)=0$, they do not depend on $s$,
which justifies the notation.
Since $\cL(H)$ is an Ulrich line bundle we have
$\chi(\cL)=\chi(\cL(-H))=0$ which gives $L^2=-4$ and $LH=g-1$, where
$L=c_1(\cL)$. Similar constraints hold for $H-L$.
In particular, $L$ and $H-L$ have the same
degree with respect to $H$, hence $\hh^0(\cO_X(2L-H)) \le 1$, with equality being
attained if and only if $L \equiv H-L$. Likewise, $\hh^2(\cO_X(2L-H))
= \hh^0(\cO_X(H-2L)) \le 1$.
Now we observe the following bound:
\begin{align*}
\ext_X^1(\cL^*(H),\cL)& =\hh^1(\cO_X(2L-H)) = \\
& = \hh^0(\cO_X(2L-H))+\hh^2(\cO_X(2L-H))-\chi(\cO_X(2L-H)) \le \\
& \le 2-\chi(\cO_X(2L-H)) = g+7,
\end{align*}
the last equation being obtained by Riemann-Roch after plugging
$L^2=-4$ and $HL=g-1$. In view of the rigidity of $H-L$ and $L$, the family of
sheaves appearing as an extension \eqref{extension} is parametrized by
$\PP(\Ext^1_X(\cL^*(H),\cL))$ and hence has dimension at most $g+6$. So this family
cannot dominate the $(2g+8)$-dimensional family $S_0$, a
contradiction.

\bigskip

It  follows from Theorem \ref{main} that $X$ is strictly Ulrich wild
in the sense of \cite{daniele-joan:wild}.
 The next result refines this fact in terms of moduli spaces. It was proved when $\Pic(X)$ is
generated by $H$ in \cite[Theorem 2.7]{aprodu-ortega-farkas:crelle}. A
modification of that argument allows to prove the result in general.

\begin{thm}
  Let $X$ be a K3 surface and $H$ be a very ample line bundle on
  $X$. Then, for any positive integer $r$, the moduli space
  $\rM_X(2r,r H,-2r)$ is of dimension $2(r^2(g+3)+1)$. Given a general sheaf
  $\cF$ in this space, $\cF(H)$ is a stable Ulrich bundle.
\end{thm}

\begin{proof}
  Given a coherent sheaf $\cE$ or rank $r > 0$ on $X$ we write
  $\rP(\cE) \in \QQ[t]$ for the Hilbert polynomial of $\cE$ and
  $\rp(\cE)$ for its reduced version, namely
  $\rP(\cE)=\chi(\cE(tH))$ and $\rp(\cE)=\rP(\cE)/r$. We put
  $\rp_0=(g-1)(t+1)t$ so that, if $\cE$ is an Ulrich
  sheaf, then $\rp(\cE(-H))=\rp_0$. Note that, if $\cE_1$ and $\cE_2$ are
  non-isomorphic stable sheaves with $\rp(\cE_1)=\rp(\cE_2)$, then
  $\Ext^k_X(\cE_i,\cE_j)=0$ for $k=0,2$ and $i \ne j$.

  The proof goes by induction on $r$, the case $r=1$ being given by
  Corollary \ref{stable-rank-2}. For $r \ge 1$, we select a stable
  bundle $\cE_2$ in $\rM_X(2r,rH,-2r)$ given by the induction
  hypothesis and a stable
  bundle $\cE_1$ in $\rM_X(2,H,-2)$, with $\cE_i(H)$ Ulrich for
  $i=1,2$, taking care  that 
  $\cE_1$ is not isomorphic to $\cE_2$ for $r=1$. This is of course
  possible since $\dim(\rM_X(2,H,-2))>0$. This way we have:
  \begin{align}
    \label{ext2} 
    & \Ext^k_X(\cE_i,\cE_j)=0, && \mbox{for $k=0,2$ and $i \ne j$}, \\
    \label{ext1} & \ext^1_X(\cE_i,\cE_j)=2r(g+3) && \mbox{for $i \ne j$.}
\end{align}

Note that, for any choice of $\zeta \in \PP(\Ext^1_X(\cE_2,\cE_1))$, the
sheaf $\cE^\zeta$ fitting as middle term of the associated extension
is a locally free semistable sheaf, with $\cE^\zeta(H)$ (as extension
of sheaves having these properties).
By direct computation, we see that it lies $\rM_X(2(r+1),(r+1)
H,-2(r+1))$. Of course this sheaf is not stable, as $\cE_1$ is a
sub-sheaf of $\cE^\zeta$ with quotient $\cE_2$ and the reduced
Hilbert polynomial of all these sheaves is $\rp_0$.
However, it follows by \cite[Theorem A, ii)]{daniele-joan:wild} that
$\cE^\zeta$ is simple, as the representation of the associated
Kronecker consists of a single non-zero map of one-dimensional vector
spaces, and as such it is simple. Alternatively one may apply
\cite[Proposition 5.3]{pons_llopis-tonini}.
  
  We record the defining
  sequence:
  \begin{equation}
    \label{zeta}
    0 \to \cE_1 \to \cE^\zeta \to \cE_2 \to 0.
  \end{equation}

  In the same spirit as in Lemma \ref{deform}, we take a deformation
  of $\cE^\zeta$ in the space of simple sheaves, which is unobstructed
  of dimension $2((r+1)^2(g+3)+1)$ at $\cE^\zeta$. We consider thus an
  integral quasi-projective variety $S$ as base of an $S$-flat family of
  simple sheaves $\cF_s$ with $\cF_s(H)$ Ulrich for all $s$ and $\cF_{s_0} \simeq \cE^\zeta$ for some
  $s_0 \in S$, the base $S$ being locally isomorphic to the moduli
  space of simple 
  sheaves around the point $s_0$. We may assume that $\cF_s$ is locally free
  for all $s\in S$.

  \medskip

  \begin{claim} \label{ilclaim}
   There is an open dense subset $S_0$ of $S$ such
  that, for any stable sheaf $\cK$ with $\rk(\cK) < 2(r+1)$, $\rk(\cK)
  \ne 2$ and
  $\rp(\cK)=\rp_0$,
  we have $\Hom_X(\cK,\cF_s)=0$, for all $s \in S_0$.
  \end{claim}
  
  \begin{proof}[Proof of the claim]
  Clearly 
  it suffices to find such open subset for a fixed rank $u$ of $\cK$
  and take the intersection of the corresponding open subsets for all
  $u < 2(r+1)$, $u \ne 2$.

  So let $\rN$ be the moduli space of stable sheaves $\cE$ on $X$ with Hilbert
  polynomial $\rP(\cE)=u \rp_0$.
  Let $\cU$ be a quasi-universal family over $X \times \rsta$,
  cf. \cite[Proposition 4.6.2]{huybrechts-lehn:moduli} and denote by
  $\sigma$ and $\pi$ the projection maps $X \times \rsta \to \rsta$ and $X
  \times \rsta \to X$, respectively.

  For $y \in \rsta$ let $\cU_y$ be the corresponding sheaf over
  $X$. We observe that, 
  applying $\Hom_X(\cU_y,-)$ to
  \eqref{zeta}, using the definition of $\rN$ and $\zeta$ and the fact that the $\cE_i$'s are stable with
  $\rp(\cE_i)=\rp(\cU_y)$ we get
  $\Hom_X(\cU_y,\cE^\zeta)=0$. 
  Indeed, the only case to check is for $u=2r$ when $y$ corresponds to
  the sheaf $\cE_2$, but $\Hom_X(\cE_2,\cE^\zeta)=0$, for otherwise by
  stability of $\cE_2$ the exact sequence \eqref{zeta} would split,
  contradicting our assumption on $\zeta$.

Then, Serre duality gives, for all $y \in \rN$:
  \begin{equation}
    \label{s0}
    \HH^2((\cE^\zeta)^* \otimes \cU_y)  \simeq \Ext^2_X(\cE^\zeta,\cU_y) =0.
  \end{equation}

  Now consider $X \times \rsta \times S$, put $\tau$ for the
  projection $\rsta \times S
  \to S$ and denote by $\bar \sigma$,
  $\bar \pi$, $\bar \tau$ the projection maps from $X \times \rsta \times
  S$ onto $X \times S$, $\rsta \times S$ and $X \times \rsta$,
  respectively.
  Let $\cV = \bar \pi^*(\cF^*) \otimes \bar \tau^*(\cU)$.
  Since $\cV$ is flat over the integral
  base $\rsta \times S$ and $\bar \sigma$ has relative dimension $2$,
  base-change gives, for all $(y,s) \in \rsta \times S$:
  \begin{equation}
    \label{base-change}
  \rR^2\bar \sigma_*(\cV)_{(y,s)} \simeq \HH^2(\cF_s^* \otimes \cU_y).    
  \end{equation}

  Let $W$ be the support of $\rR^2\sigma_*(\cV)$, i.e. the closed subset of points $(y,s) \in \rsta \times S$ such that
  $\rR^2\sigma_*(\cV)_{(y,s)} \ne 0$. By \eqref{s0} and
  \eqref{base-change}, we have $W \cap \rsta \times \{s_0\} =
  \emptyset$, i.e. $s_0$ does not lie in $\tau(W)$. Then there is an
  open neighbourhood $S_0 \subset S$ of $s_0$ which is disjoint from
  $\tau(W)$. Again by \eqref{base-change}, we get $\HH^2(\cF_s^*
  \otimes \cU_y)=0$ for all $(y,s) \in \rsta \times S_0$, which 
  proves the claim.
\end{proof}

Let us now conclude the proof of the theorem. In view of the claim, we
have two alternatives for
$s$ generic in $S_0$: either $\Hom(\cK,\cF_s)=0$ for any stable sheaf $\cK$ with $\rk(\cK) < 2(r+1)$ and
  $\rp(\cK)=\rp_0$, or otherwise this happens for all such $\cK$ 
  except for $\rk(\cK)=2$ and there actually exists a stable $\cK$ in
  $\rsta$ such that $\Hom(\cK,\cF_s) \ne 0$.

  In the first alternative $\cF_s$ is stable, so we assume
  that the second one takes place and look for a contradiction. 
  We go back to  Claim \ref{ilclaim} and carry out the
  same argument for $u=2$, with $y_0$ being the point corresponding to
  $\cE_1$. Observe that $\cK$ must lie in $\rM_X(2,H,-2)$ as the proof of
  Claim
  \ref{ilclaim} applies verbatim on any other component of $\rsta$.

We note that $W \cap \rsta \times \{s_0\}=\{(y_0,s_0)\}$,
  as clearly $\Hom_X(\cK,\cE^\zeta)=0$ for all $\cK$ in $\rsta
  \setminus \{y_0\}$. So $W$ is properly contained in $\rsta\times S$. Moreover, we easily have
  $\hom_X(\cE_1,\cE^\zeta)=1$. Recall by construction of the
  quasi-universal family that there is $u_0$ such
  that $\rk(\cU)=2u_0$ and that, for $y \in \rsta$, the sheaf $\cU_y$ is a
  direct sum of $u_0$ copies of the stable sheaf of rank $2$ in
  $\rM_X(2,H,-2)$ corresponding to $y$.
  Therefore, the sheaf $\rR^2\bar \sigma_*(\cV)_{(y,s)}$ has rank at
  least $u_0$ at any $(y,s) \in W$, and rank precisely $u_0$ at $(y_0,s_0)$.
  So there is an open dense subset $W_0$ of $W$ where $\rR^2\bar
  \sigma_*(\cV)$ is free of rank $u_0$. For any $(y,s) \in W_0$, the
   stable sheaf $\cK$ corresponding to $y$ satisfies
   $\hom_X(\cK,\cF_s)=1$; up to proportionality we have thus a unique
   non-zero map $\eta_{y,s} : \cK \to \cF_s$.
   Stability easily impies that $\eta_{y,s}$ 
  is injective, so there is an exact
  sequence:
  \[
  0 \to \cK \to \cF_s \to \cK' \to 0,
  \]
  for a well-defined sheaf $\cK'=\cok(\eta_{y,s})$, for all $(y,s) \in W_0$.

  For $s=s_0$ the sheaf $\cK'$ is just
  $\cE_2$ so, by openness of stability, up to shrinking $W_0$ we may
  assume that $\cK'$ is stable for all $(y,s) \in W_0$. Note that $\cK'$
  lies in $\rM(2r,rH,-2r)$.

  Under our assumption, such sequence should exist for any $s$ in an open
  neighbourhood of $s_0$. Then the family of sheaves $\cF$ should be
  dominated by the family of extensions of $\cK$ by $\cK'$ as $s$
  varies around $s_0$. We see that the
  dimension of this family of extensions 
is:
  \[
   \dim(\rM_X(2,H,-2))+\dim(\rM_X(2r,rH,-2r))+\dim(\PP\Ext^1_X(\cK',\cK)),
  \]
  which equals $2(r(r+1)+1)(g+3)+3$, as it follows by formulas
  \eqref{ext2}, \eqref{ext1} applied to $\cK$ and $\cK'$ instead
  of $\cE_1$ and $\cE_2$. On the other hand, the dimension
  of $S$ is $2((r+1)^2(g+3)+1)$. The difference of these
  dimensions is $2r(g+3)-1$ and since this is always positive for
  $r \ge 1$, $g\ge 3$, we get that the family of simples sheaves appearing as extensions cannot be dense in
  $S_0$. This contradiction concludes the proof.
  \end{proof}

The previous result is in some sense optimal as general
K3 surfaces do not support Ulrich bundles of odd rank,
cf. \cite[Corollary 2.2]{aprodu-ortega-farkas:crelle}.

\begin{rmk}
  An argument similar to the one of Theorem \ref{main} has been used to
construct ACM and Ulrich bundles on Fano threefolds of index $1$. Indeed, it
follows from the main result of \cite{brafa2} that any smooth Fano threefold of Picard number
$1$ and index $1$, containing a line $L$ with normal bundle $\cO_L \oplus
\cO_L(-1)$ (such a
threefold was called ``ordinary'' in that paper) admits an Ulrich
bundle of rank $2$. Ulrich sheaves of rank $2$ are precisely ACM
sheaves $\cE$ with $c_1(\cE(-H))=H$ and $c_2(\cE(-H))=(g+3)L$, where
$L \subset X$ is a line. We do not
 know if the same result holds for non-ordinary threefolds.
\end{rmk}

\begin{rmk}
  Theorem \ref{main} implies for instance that any integral quartic surface  supports
  an Ulrich bundle of rank $2$. 
  If $X$ is not integral, then $X$ must the union of (possibly multiple) surfaces of
  degree $\ge 3$.
  For each component it is possible to find a rank-$2$
  Ulrich bundle, we refer to \cite[Lemma
  7.2]{daniele-joan:wild} for the slightly delicate case 
  of singular cubic surfaces. This yields existence of an Ulrich sheaf of rank $2$
  on an arbitrary quartic surface.

 However the resulting sheaf will fail to
  be locally free over the intersection of the components. Finding
  locally free Ulrich sheaves of rank $2$ seems more tricky when $X$
  is not irreducible and might be impossible when $X$ is not reduced. To justify this let us mention that, for
  instance if $X$ the union of two distinct double planes, the rank of
  any locally free Ulrich sheaf on $X$
  must be a multiple of $4$ by \cite[Proposition
  4.14]{ballico-huh-malaspina-pons-llopis-to-appear}.
\end{rmk}

\medskip

I would like to thank M. Aprodu, G. Casnati, A. Perego, J. Pons-Llopis and
P. Stellari for useful discussions.
I am grateful to the referee for useful remarks.

 \bibliography{ulrich-K3.final.bib}

\def\cprime{$'$} \def\cprime{$'$} \def\cprime{$'$} \def\cprime{$'$}
  \def\cprime{$'$} \def\cprime{$'$} \def\cprime{$'$} \def\cprime{$'$}
  \def\cprime{$'$}
\begin{thebibliography}{BHMP16}

\bibitem[AFO17]{aprodu-ortega-farkas:crelle}
Marian Aprodu, Gavril Farkas, and Angela Ortega.
\newblock Minimal resolutions, {C}how forms and {U}lrich bundles on {$K3$}
  surfaces.
\newblock {\em J. Reine Angew. Math.}, 730:225--249, 2017.

\bibitem[AK80]{altman-kleiman:compactifying}
Allen~B. Altman and Steven~L. Kleiman.
\newblock Compactifying the {P}icard scheme.
\newblock {\em Adv. in Math.}, 35(1):50--112, 1980.

\bibitem[Art90]{artamkin:torsion-free}
I.~V. Artamkin.
\newblock Deformation of torsion-free sheaves on an algebraic surface.
\newblock {\em Izv. Akad. Nauk SSSR Ser. Mat.}, 54(3):435--468, 1990.

\bibitem[Bea16]{beauville:abelian-ulrich}
Arnaud Beauville.
\newblock Ulrich bundles on abelian surfaces.
\newblock {\em Proc. Amer. Math. Soc.}, 144(11):4609--4611, 2016.

\bibitem[Bea18]{beauvill:ulrich-intro}
Arnaud Beauville.
\newblock An introduction to {U}lrich bundles.
\newblock {\em Eur. J. Math.}, 4(1):26--36, 2018.

\bibitem[BF11]{brafa2}
Maria~Chiara Brambilla and Daniele Faenzi.
\newblock {Moduli spaces of rank-2 ACM bundles on prime Fano threefolds.}
\newblock {\em Mich. Math. J.}, 60(1):113--148, 2011.

\bibitem[BHMP16]{ballico-huh-malaspina-pons-llopis-to-appear}
Edoardo {Ballico}, Sukmoon {Huh}, Francesco {Malaspina}, and Joan
  {Pons-Llopis}.
\newblock {ACM sheaves on the double plane}.
\newblock {\em ArXiv e-print math.AG/1604.00866}, 2016.
\newblock To appear in Trans. Amer. Math. Soc.

\bibitem[Cas17]{casnati:special-non-special}
Gianfranco Casnati.
\newblock Special {U}lrich bundles on non-special surfaces with {$p_g=q=0$}.
\newblock {\em Internat. J. Math.}, 28(8):1750061, 18, 2017.

\bibitem[CKM12]{coskun-kulkarni-mustopa:quartic}
Emre Coskun, Rajesh~S. Kulkarni, and Yusuf Mustopa.
\newblock Pfaffian quartic surfaces and representations of {C}lifford algebras.
\newblock {\em Doc. Math.}, 17:1003--1028, 2012.

\bibitem[Cos17]{coskun:ulrich-survey}
Emre Coskun.
\newblock A survey of {U}lrich bundles.
\newblock In {\em Analytic and algebraic geometry}, pages 85--106. Hindustan
  Book Agency, New Delhi, 2017.

\bibitem[ESW03]{eisenbud-schreyer-weyman}
David Eisenbud, Frank-Olaf Schreyer, and Jerzy Weyman.
\newblock Resultants and {C}how forms via exterior syzygies.
\newblock {\em J. Amer. Math. Soc.}, 16(3):537--579, 2003.

\bibitem[FP15]{daniele-joan:wild}
Daniele {Faenzi} and Joan {Pons-Llopis}.
\newblock {The Cohen-Macaulay representation type of arithmetically
  Cohen-Macaulay varieties}.
\newblock {\em ArXiv e-print math.AG/1504.03819}, 2015.

\bibitem[HL97]{huybrechts-lehn:moduli}
Daniel Huybrechts and Manfred Lehn.
\newblock {\em The geometry of moduli spaces of sheaves}.
\newblock Aspects of Mathematics, E31. Friedr. Vieweg \& Sohn, Braunschweig,
  1997.

\bibitem[{Kol}08]{kollar:hulls-husks}
Janos {Koll{\'a}r}.
\newblock {Hulls and Husks}.
\newblock {\em ArXiv e-print math.AG/0805.0576}, 2008.

\bibitem[Muk84]{mukai:symplectic}
Shigeru Mukai.
\newblock Symplectic structure of the moduli space of sheaves on an abelian or
  {$K3$} surface.
\newblock {\em Invent. Math.}, 77(1):101--116, 1984.

\bibitem[PLT09]{pons_llopis-tonini}
Joan Pons-Llopis and Fabio Tonini.
\newblock A{CM} bundles on del {P}ezzo surfaces.
\newblock {\em Matematiche (Catania)}, 64(2):177--211, 2009.

\bibitem[Qin93]{qin:simple}
Zhenbo Qin.
\newblock Moduli of simple rank-{$2$} sheaves on {$K3$}-surfaces.
\newblock {\em Manuscripta Math.}, 79(3-4):253--265, 1993.

\bibitem[Ros52]{rosenlicht:equivalence}
Maxwell Rosenlicht.
\newblock Equivalence relations on algebraic curves.
\newblock {\em Ann. of Math. (2)}, 56:169--191, 1952.

\bibitem[SD74]{saint-donat:K3}
Bernard Saint-Donat.
\newblock Projective models of {$K-3$} surfaces.
\newblock {\em Amer. J. Math.}, 96:602--639, 1974.

\bibitem[SE10]{eisenbud-schreyer:betti-numbers}
Frank-Olaf Schreyer and David Eisenbud.
\newblock Betti numbers of syzygies and cohomology of coherent sheaves.
\newblock In {\em Proceedings of the {I}nternational {C}ongress of
  {M}athematicians. {V}olume {II}}, pages 586--602. Hindustan Book Agency, New
  Delhi, 2010.

\end{thebibliography}
%\bibliography{bibliography.bib}

\bibliographystyle{alpha}

\end{document}